\documentclass[12pt]{article}
\usepackage{geometry}
\usepackage{setspace}
\usepackage{amsthm,amsmath,amssymb}
\usepackage{mathrsfs}
\usepackage{accents}
\usepackage{color}
\usepackage{graphicx}
\usepackage{subfigure}
\usepackage{array}
\usepackage{caption}
\usepackage{cite}
\usepackage{hyperref}
\usepackage{enumerate}

\newtheorem{theorem}{Theorem}[section]

\newtheorem{lemma}[theorem]{Lemma}

{\theoremstyle{definition}}

\theoremstyle{definition}

\numberwithin{equation}{section}
\numberwithin{figure}{section}

\allowdisplaybreaks[4]

\title{Excursions in Sylvester-Gallai land}
\author{Imre B\'ar\'any$^{1,2,3}$  \ \ Julia Q. Du$^{1,4,5}$ \ \ \ Dan Schwarz$^\dagger$\\
Liping Yuan\footnote{Corresponding author. Email:  barany.imre@renyi.hu (I. B\'ar\'any); qddu@hebtu.edu.cn (J. Q. Du); lpyuan@hebtu.edu.cn (L. Yuan); tuzamfirescu@gmail.com (T. Zamfirescu)} $^{1,4,5}$ \ \ \ Tudor Zamfirescu$^{1,6,7}$}

\date{\today}

\begin{document}
\large

\maketitle
{\small 1. School of Mathematical Sciences,
Hebei Normal University,
050024 Shijiazhuang, P.R. China.

2. Alfr\'ed R\'enyi Institute of Mathematics, HUN-REN
13 Re\'altanoda Street, Budapest 1053, Hungary.

3. Department of Mathematics, University College London, London, UK.

4. Hebei Key Laboratory of Computational Mathematics and Applications, 050024 Shijiazhuang,  P.R. China.

5. Hebei Research Center of the Basic Discipline Pure Mathematics,
050024 Shijiazhuang, P.R. China.

6. Fachbereich Mathematik, Technische Universit\"at Dortmund,
44221 Dortmund, Germany.

7. Roumanian Academy,  014700 Bucharest, Roumania.
}

\textbf{Abstract}: The Sylvester-Gallai theorem states that for a finite set of points in the plane, if every line determined by any two of these points also contains a third, then the set is necessarily made of
collinear points. In this paper, we first provide a counterexample in the plane when the point set is countably infinite but bounded. Then we consider a variant of the Sylvester-Gallai theorem where instead of a finite point set we have a finite family of convex sets in $\mathbb{R}^d$ ($d\geq 2$). Finally, we present another variant of the Sylvester-Gallai theorem,  when instead of point sets we have a finite family of line-segments in the plane.



\section{Introduction and main results}\label{sec:intro)}

In 1893, Sylvester \cite{Syl93} posed the following question:
Let a finite set of points in the plane
have the property that the line through
any two of them passes through
a third point of the set.
Must all the points lie on one
line?
No solution was offered at that time and the problem seemed to have been forgotten. Forty years later it resurfaced as a conjecture by Erd\H{o}s \cite{Erd43}. Gallai \cite{Gal44} gave the first proof. In fact Melchior \cite{Mel41} proved the same result in dual form earlier than Gallai but probably was unaware of Sylvester's question. The Sylvester-Gallai theorem states the following.

\begin{theorem}\label{Syl-thm}
A finite set of points in the plane, such that every line determined
by any two of these points also contains a third,
is necessarily made of collinear points.
\end{theorem}

This theorem admits one of the most beautiful proofs ever (see \cite{AZ18}). It is due to Kelly and appeared in a paper by Coxeter~\cite{Cox48} in 1948.
For infinite sets, the statement is wrong. $\mathbb{R}^2$ is an immediate counterexample; the $\mathbb{Z}^2$ lattice is countably infinite
(still unbounded), and also a counterexample. As any infinite, bounded set of points in the plane has necessarily (at least) one accumulation point,
can we exhibit such a counterexample, minimal in the sense of having exactly one accumulation point?

Theorem \ref{one-limit-p-ex} answers this affirmatively. Without the accumulation point, the set we construct is made of isolated points only;
with the accumulation point, the set becomes compact.

In Sections 3 and 4 we consider a variant of the Sylvester-Gallai theorem where instead of a finite point set we have a finite family of convex sets. Precisely, let $\mathcal F=\{K_1,\ldots,K_n\}$ be a finite family of pairwise disjoint, compact and strictly convex sets in $\mathbb{R}^d$ where each $K_i$ has nonempty interior. We assume $n\ge 2$ and $d\ge 2$, the case $d=1$ is not interesting and so is the case $n=1$. We will call such a family a {\sl convex system}. Define $U:=\bigcup_{i=1}^n K_i$ and call a line $L$ {\sl ordinary} if $L \cap U$ consists of exactly two points, the usual term for such a line. Of course the two points in $L \cap U$ belong to distinct $K_i$. We will show in Theorem~\ref{th:sik} that for $d=2$ there are convex systems with no ordinary line if and only if $n > 3$. We prove further in Theorem~\ref{th:high} that every convex system with $d\ge 3$ and $n\ge 2$ has an ordinary line.

Section \ref{sec:segments} is about another variant of the Sylvester-Gallai theorem, namely when instead of point sets we have a finite family of line-segments in the plane.

\section{ A countably infinite, bounded counterexample}\label{sec:constr}

Line-segments in this paper are topologically closed.
For any $x,y\in\mathbb{R}^d$,
let $xy$ be the line-segment joining $x$ and $y$,
$\overline{xy}$ be the line through $x$ and $y$.
For $M\subset\mathbb{R}^d$, $\overline{M}$ denotes the affine hull of $M$ and $\mathrm{conv}\,{M}$ its convex hull.

In order to prove Theorem \ref{one-limit-p-ex},
we need the following lemma.

\begin{lemma}\label{bxy}
The length $b$ of the angle bisector of
the $2\pi/3$ angle of a triangle,
made by adjacent sides of lengths $x$ and $y$,
is given by the formula
\begin{align*}
\frac{1}{b}=\frac{1}{x}+\frac{1}{y}.
\end{align*}
\end{lemma}

\begin{proof}
The simple two-way calculation of the area of the triangle
\begin{align*}
\frac{1}{2}xy\sin\frac{2\pi}{3}=\frac{1}{2}xb\sin\frac{\pi}{3}
+\frac{1}{2}by\sin\frac{\pi}{3}
\end{align*}
yields the desired formula.
\end{proof}

A non-collinear, infinite, bounded set of points in $\mathbb{R}^2$ meeting no line in exactly two points is easily found: take a (full) triangle and remove its vertices. But can we find a countable such set?

\begin{theorem}\label{one-limit-p-ex}
There exists a non-collinear, countably infinite, bounded set of points in the plane, meeting no line in exactly two points.
\end{theorem}

\begin{proof}
Take six half-lines $L_1, \ldots, L_6$
meeting at the origin $\mathbf{0}$
and with angular measure $\pi/3$ between any two adjacent half-lines.

Let
\begin{align*}
F&=\{f\in L_i\colon i=1,3,5, \text{ and } \|f\|=1/(3j -2) \text{ for some } j\in\mathbb{N}\},\\
G&=\{g\in L_i\colon i=2,4,6, \text{ and } \|g\|=1/(3j -1) \text{ for some } j\in\mathbb{N}\}.
\end{align*}
Clearly, $\mathbf{0}$ is the accumulation point of $H=F\cup G$.
Next we show that any line through two distinct points $p,q\in H$
also contains a third point of $H$.

If $\overline{pq}$ passes through $\mathbf{0}$,
then it is clear that
$\overline{pq}$ contains a third point of $H$.
Otherwise,
$\angle p\mathbf{0}q$ equals ${\pi}/{3}$ or ${2\pi}/{3}$.

\begin{figure}[htbp]
\centering
\includegraphics[width=7cm]{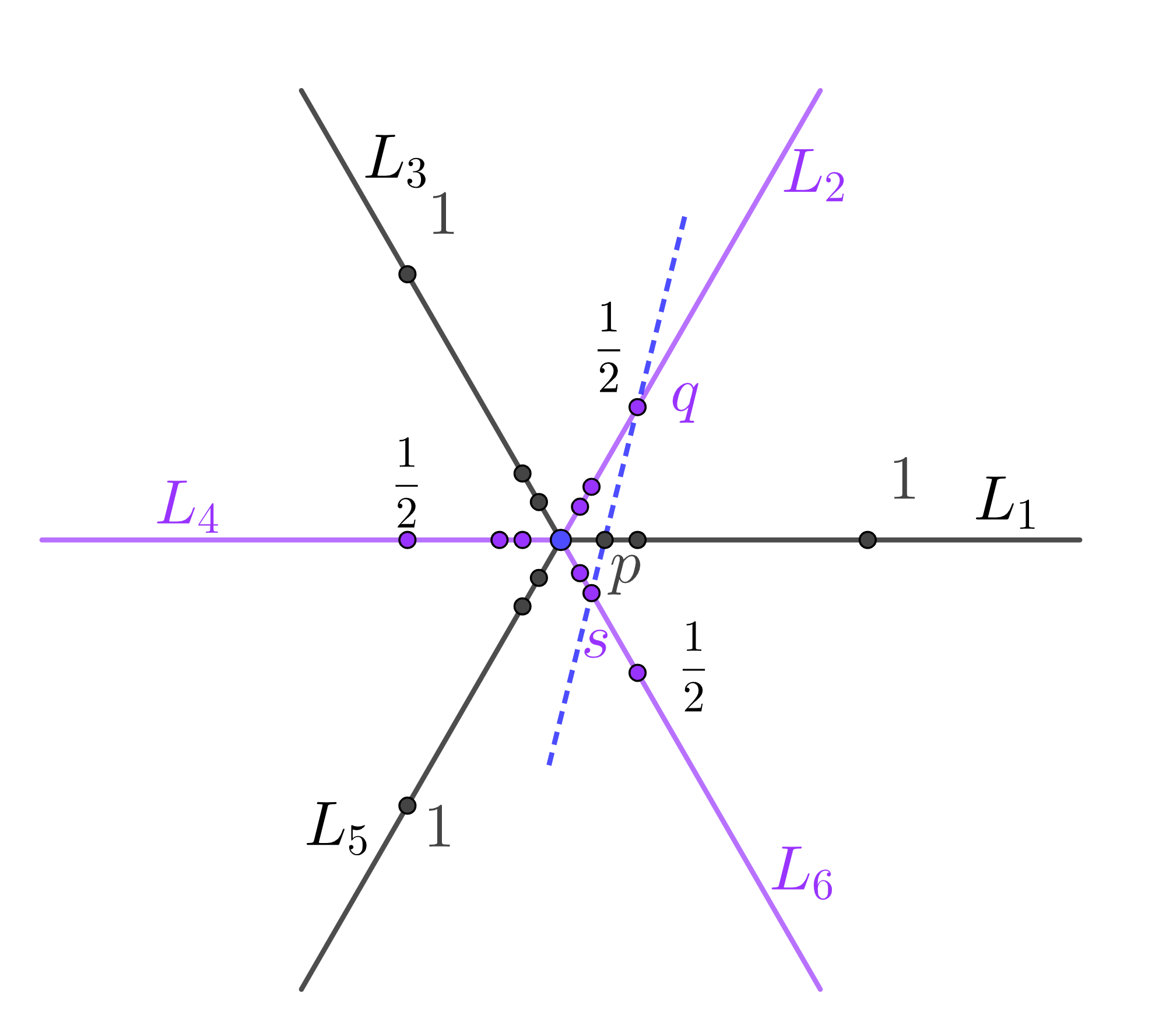}
\includegraphics[width=7cm]{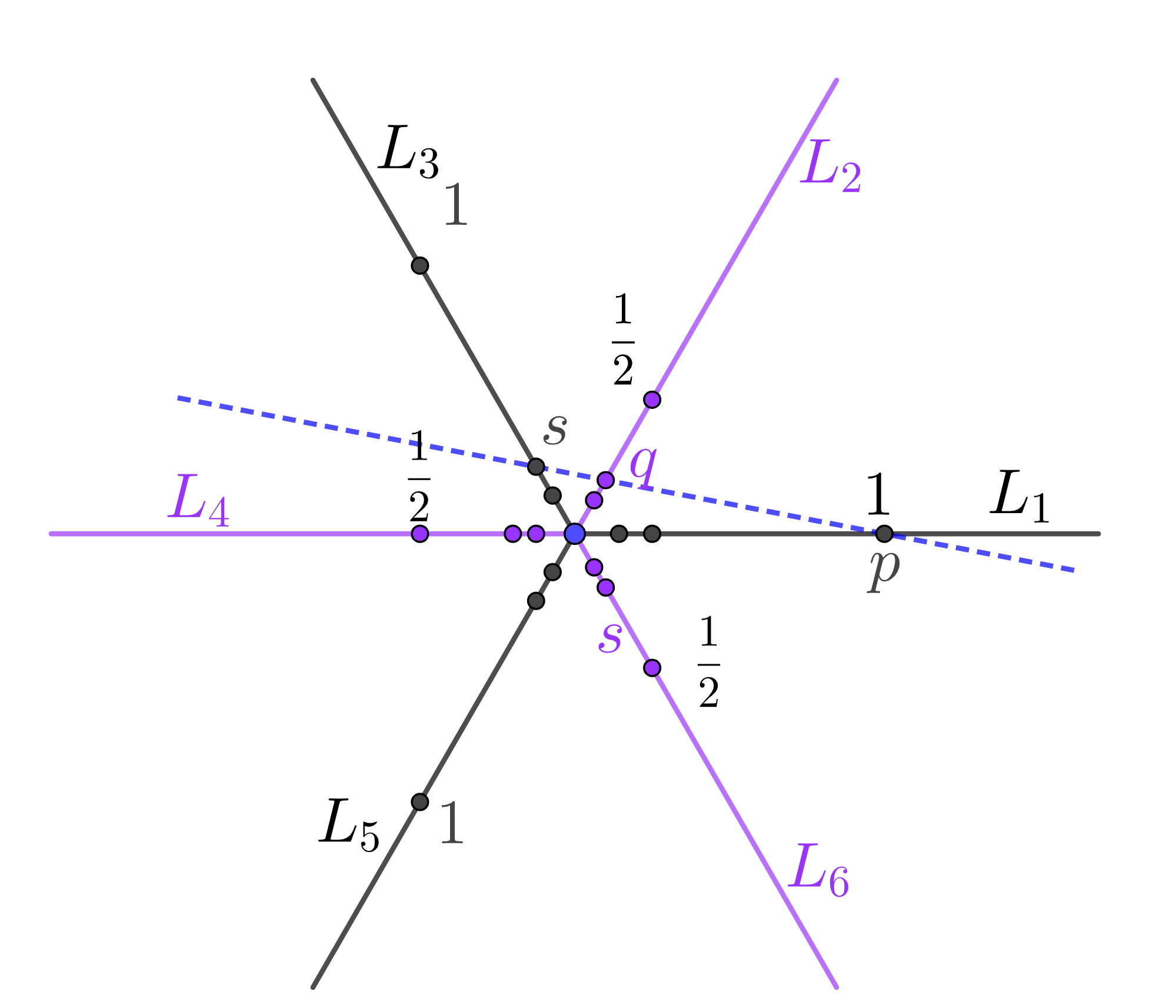}\\
$(a)$ \hspace{6cm} $(b)$
\caption{$\angle p\mathbf{0}q ={\pi}/{3}$.}
\label{case1}
\end{figure}

For $\angle p\mathbf{0}q ={\pi}/{3}$,
without loss of generality, assume that $p\in L_1$ and $q\in L_2$.
Then there exist positive integers $i,j$ such that
$\|p\|=1/(3i-2)$ and $\|q\|=1/(3j-1)$.
If $i> j$, take $s\in L_6$ such that $\|s\|=1/(3(i-j)-1)$;
see Figure \ref{case1} $(a)$.
By Lemma \ref{bxy}, we have $s\in \overline{pq}$.
If $i\leq j$, take $s\in L_3$ such that $\|s\|=1/(3(j-i+1)-2)$
(see Figure \ref{case1} $(b)$).
We have $s\in \overline{pq}$, by Lemma \ref{bxy}.

For $\angle p\mathbf{0}q ={2\pi}/{3}$, assume
that $p\in L_k$ and $q\in L_l$.
Then $k$ and $l$ have the same parity.
If both $k$ and $l$ are odd, without loss of generality,
assume $k=1$ and $l=3$.
Then there exist positive integers $i$ and $j$, such that
$\|p\|=1/(3i-2)$ and $\|q\|=1/(3j-2)$.
Take the point $s\in L_2$ such that $\|s\|=1/(3(i+j-1)-1)$
(see Figure \ref{case2} $(a)$).
It follows from Lemma \ref{bxy} that $s\in \overline{pq}$.

\begin{figure}[htbp]
\centering
\includegraphics[width=7cm]{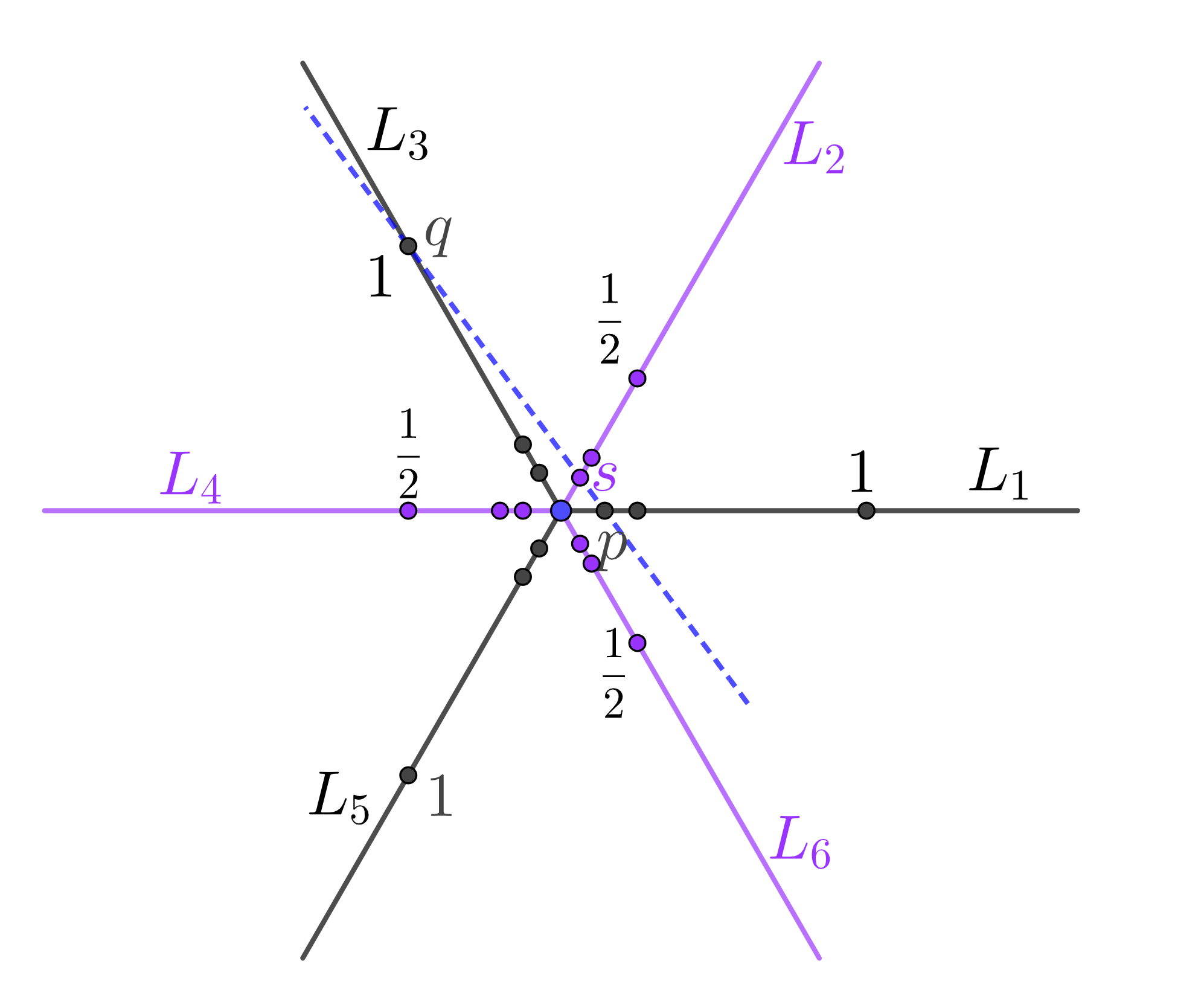}
\includegraphics[width=7cm]{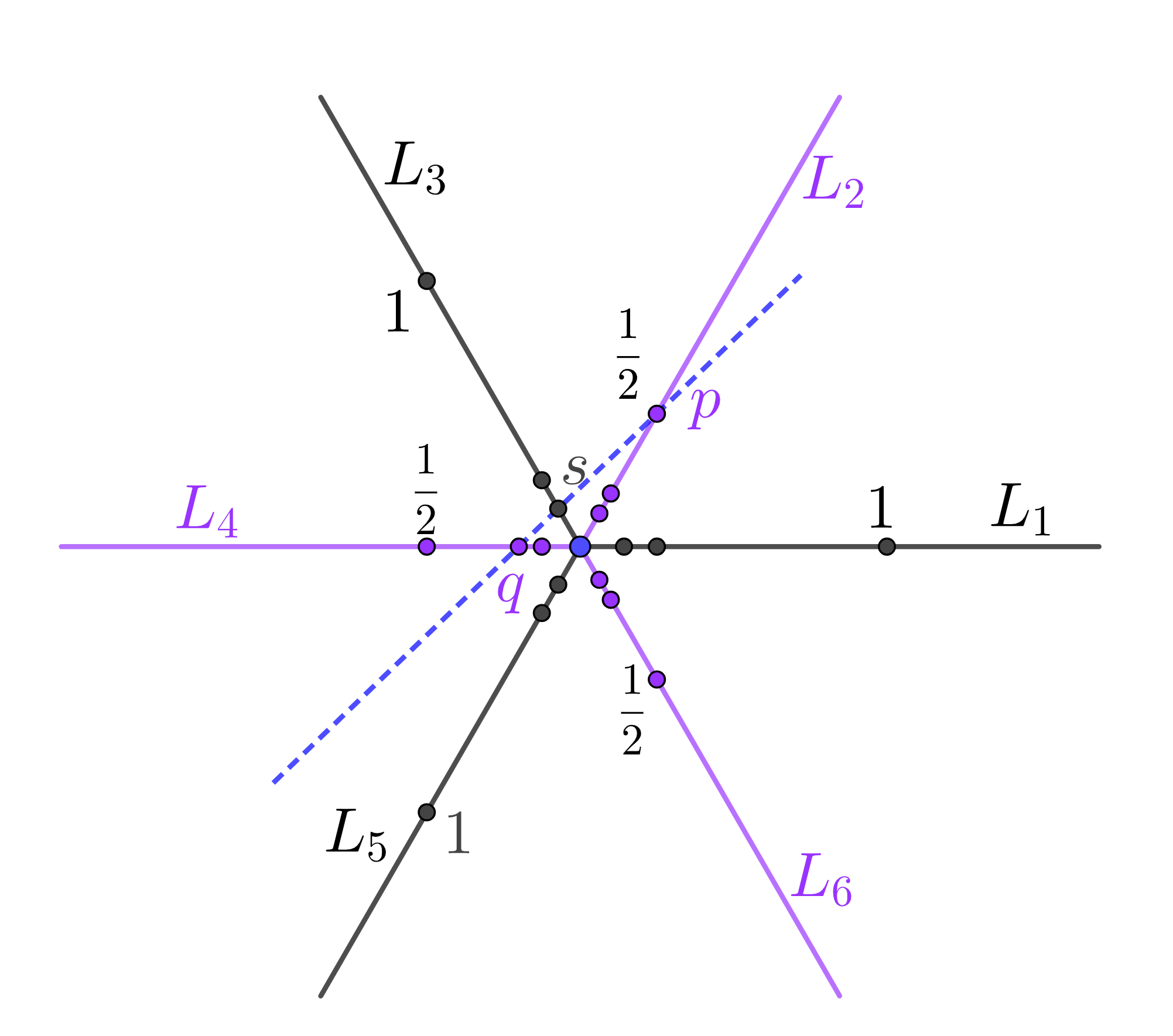}\\
$(a)$ \hspace{6cm} $(b)$
\caption{$\angle p\mathbf{0}q ={2\pi}/{3}$.}
\label{case2}
\end{figure}

If both $k$ and $l$ are even, without loss of generality,
assume $k=2$ and $l=4$.
Then there exist positive integers $i$ and $j$, such that
$\|p\|=1/(3i-1)$ and $\|q\|=1/(3j-1)$.
Take the point $s\in L_3$ such that $\|s\|=1/(3(i+j)-2)$
(see Figure \ref{case2} $(b)$).
By means of Lemma \ref{bxy},
$s\in \overline{pq}$.

The proof is complete.
\end{proof}

\section{Convex systems in the plane}\label{sec:sik}

In this section we prove the following theorem.

\begin{theorem}\label{th:sik}
There is a convex system for $d=2$ with no ordinary line if and only if $n>3$.
\end{theorem}

\begin{proof}
The case $n=2$ is trivial, the four common tangents of the two sets in $\mathcal F$ being ordinary. Recall that $U=\bigcup_{i=1}^n K_i$ and define $C=\mathrm{conv}\,U$. The boundary $\partial C$ of $C$ contains sets of the form $K_i\cap \partial C$ and when $n=3$ each such set has at most two connected components and each such component is either an arc or a point, Figure~\ref{fig:d2} shows these cases. There is a segment $xy \subset \partial C$ with $x,y \in U$ but no other point of this segment in $U$. Then the line $\overline{xy}$ is an ordinary line unless we encounter the rightmost case of Figure~\ref{fig:d2}. In that case two out of the four common tangents to $K_1$ and $K_2$ (or to $K_2$ and $K_3$) are ordinary lines.

\begin{figure}[htbp]
\centering
\includegraphics[width=14cm]{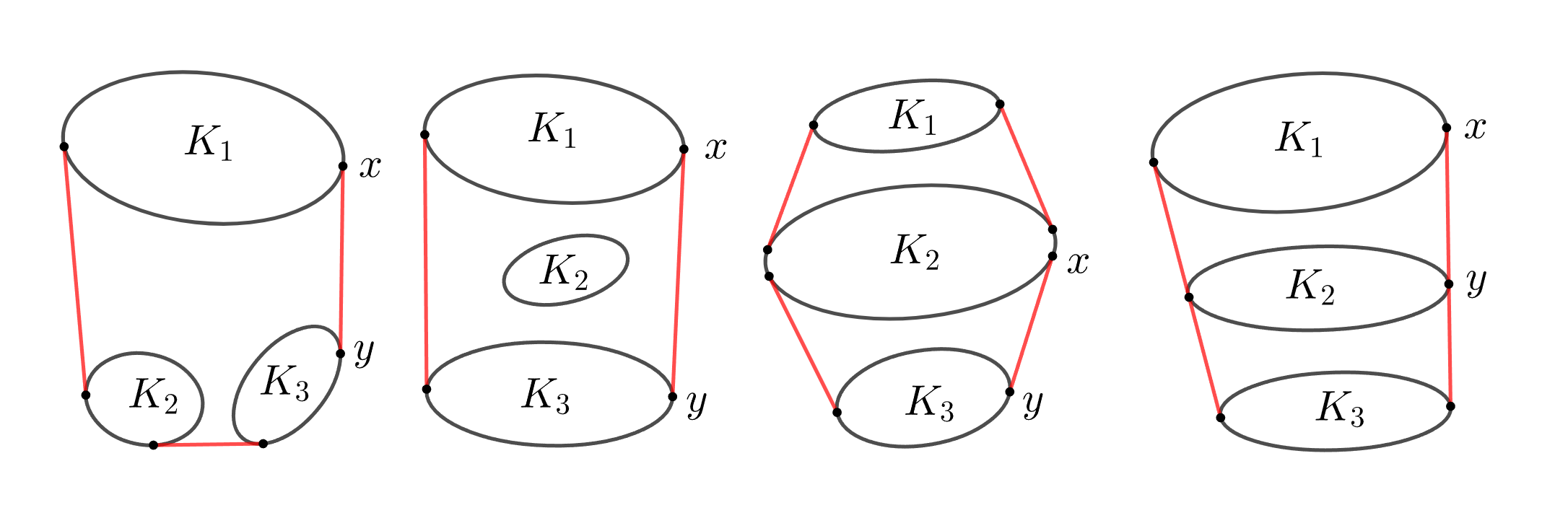}
\caption{The four cases when $n=3$.}
\label{fig:d2}
\end{figure}

Figure~\ref{fig:n>3} shows the construction for $n>3$. Some explanation is needed.
The points $a_1,b_1$ are in $K_1$ and the points $a_3,b_3$ are in $K_3$. The line-segment $a_1b_3$ contains a single point $c_4 \in K_4$ and, similarly,  $a_3b_1$ contains a single point $c_2 \in K_2$. So the lines $\overline{a_1b_3}$ and $\overline{a_3b_1}$ are  not ordinary lines.

\begin{figure}[htbp]
\centering
\includegraphics[width=8cm]{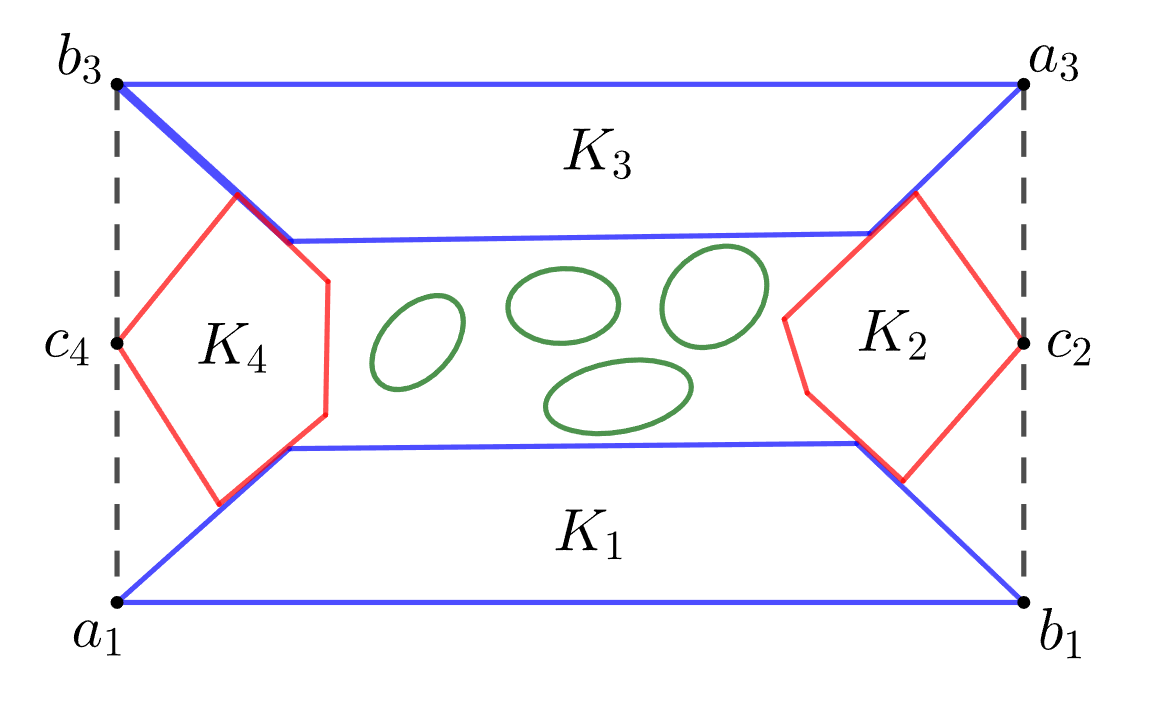}
\caption{The construction for $n>3$.}
\label{fig:n>3}
\end{figure}

The sets $K_1,K_3$ look like convex quadrilaterals and $K_2,K_4$ look like pentagons. But in fact their ``edges" are circular arcs of circles with very large radii. Moreover, the ``common edge" of $K_1$ and $K_2$ is in fact not a common edge, just $K_1$ and $K_2$ are very close. Consequently,  the corresponding two common (and separating) tangents are very close to each other and they inevitably intersect $K_3$. The case of ``common edges" of $K_2,K_3$ and $K_3,K_4$ and $K_4,K_1$ are treated similarly. There are $n-4$ sets $K_i \in \mathcal {F}$ in the middle region. It is evident that they can be chosen so that the system has no ordinary line whatsoever.
\end{proof}

\section{Convex systems in $\mathbb{R}^d$, $d>2$}\label{sec:sik}

\begin{theorem}\label{th:high}
For $d>2$ every convex system with $n\ge 2$ has an ordinary line.
\end{theorem}

\begin{proof} The cases $n=2,3$ are very simple: just take a 2-dimensional plane $P$ that intersects every $K_i \in \mathcal F$ in its interior. Then the family $\mathcal F^2:=\{K_i\cap P: K_i\in \mathcal F\}$ is a convex system in the plane $P$, and Theorem~\ref{th:sik} applies.

Next comes the case $d=3$ and $n>3$. It is evident that there is a point $z \in \partial C \setminus U$. Let $P$ be the supporting plane to $C$ at $z\in \partial C$ with outer unit normal $u$. As every $K_i\in \mathcal F$ is strictly convex and 3-dimensional, $K_i\cap P$ is either the emptyset or a single point. We may assume that $a_i$ is the single point in $K_i \cap P$ for $i=1,\ldots,m$ and $K_i\cap P=\emptyset$ for $i>m$.

Then $z \in \mathrm{conv}\, \{a_1,\ldots,a_m\}$ but $z \notin U$ implying that $m\ge 2$. If $m=2$ then $\overline{a_1a_2}$ is an ordinary line and we are done. So $m>2$. By the Sylvester-Gallai theorem if the points $a_1,\ldots,a_m$ are not collinear, then there is an ordinary line $\overline{a_ia_j}$ and we are done again.

\begin{figure}[htbp]
\centering
\includegraphics[width=11cm]{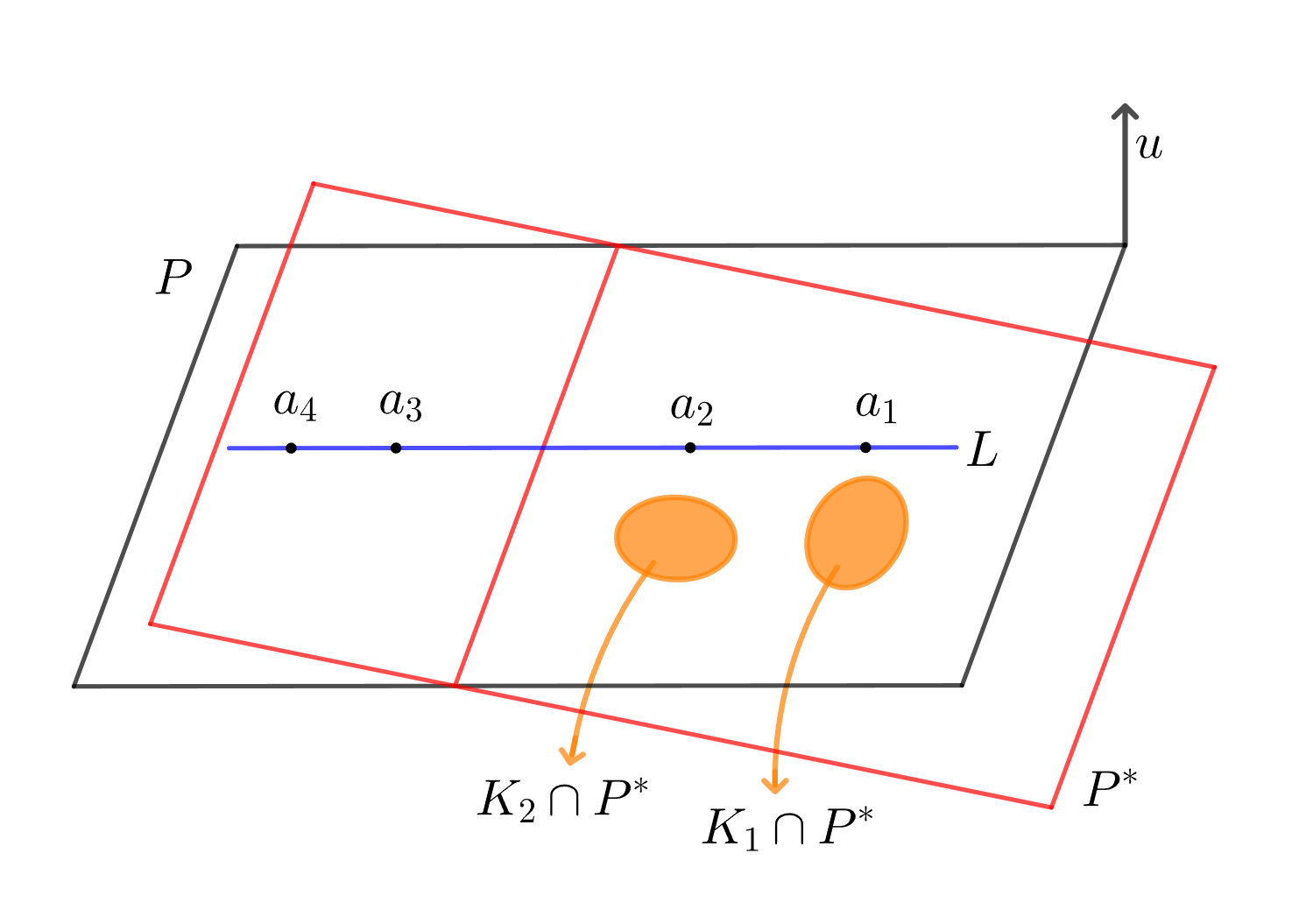}
\caption{$P$ and $P^*$ in case $d=3$.}
\label{fig:d=3}
\end{figure}

So the points $a_1,\ldots,a_m$ are all on a line $L$ and we may assume that they come in the order $a_1,\ldots,a_m$. Let $P^*$ be the 2-dimensional plane with outer normal $u+\varepsilon(a_1-a_2)$ that contains the point $\frac 12(a_2+a_3)$; see Figure~\ref{fig:d=3}. For $\varepsilon>0$ small enough the plane $P^*$ intersects $K_1$ and $K_2$ and no other $K_i$. The sets $K_1\cap P^*$ and $K_2\cap P^*$ form a convex system in $P^*$ and Theorem~\ref{th:sik} guarantees the existence of an ordinary line $\overline{b_1b_2}$ with $b_i \in K_i\cap P^*$ (in $P^*$). This line is an ordinary line for the original family $\mathcal F$ as well.

The same proof works for $d>3$ as well. The only difference is that in the supporting hyperplane $P$ to $C$ at $z \in \partial C \setminus U$ one has to use the $(d-1)$-dimensional version of the Sylvester-Gallai theorem, that can be proved for instance by Kelly's method~\cite{Cox48}.
\end{proof}

We remark that there is another proof for $d>3$. Namely, choose a point $a_i$ in the interior of $K_i$ for $i=1,2,3,4$ and let $H$ be the 3-dimensional affine subspace spanned by these points. One can check that for a suitable choice of these four points the family $\mathcal {F}^3$ consisting of the nonempty sets $H\cap K_i$  ($ i\in \{1, \ldots, n\}$) is a convex system in $H$. So the 3-dimensional case applies and gives an ordinary line for $\mathcal {F}^3$ in $H$ which is an ordinary line for $\mathcal {F}$ as well. We omit the details.

\section{Line-segments in the plane}\label{sec:segments}

\begin{theorem}\label{line-Sym}
Let $\mathcal{F}$ be a family of at most $5$ line-segments
with pairwise disjoint relative interiors in $\mathbb{R}^2$
and $M\supset \bigcup \mathcal{F}$,
such that $M\setminus \bigcup \mathcal{F}$
is finite.
If no line meets $M$ in exactly two points, then $M$
is collinear.
\end{theorem}

\begin{proof}
Suppose that $M$ is not collinear.
Then $\bigcup \mathcal{F}$ is not collinear.
Indeed, this is clear for $M= \bigcup \mathcal{F}$.
For $M\neq \bigcup \mathcal{F}$,
if $\bigcup \mathcal{F}$ is collinear,
then there exists a point $b\in M\setminus\overline{\bigcup \mathcal{F}}$.
The existence of a third point in $M$
on any line $\overline{\alpha b}$ with $\alpha\in \bigcup \mathcal{F}$
yields $M\setminus \bigcup\mathcal{F}$ infinite, absurd.
So, $P=\mathrm{conv} \bigcup\mathcal{F}$ is not a line-segment
and $\mathrm{card}\mathcal{F}\geq 2$.

\underline{Claim}.
For each vertex $a$ of $P$, there exist three
line-segments from $\mathcal{F}$ having
an endpoint at $a$.

We prove the Claim.
If $a$ is the endpoint of only one line-segment from $\mathcal{F}$,
let $A\in\mathcal{F}$ be such a line-segment.
Let $L$ and $L'$ be the two sides of $P$ with $a$ as an endpoint.
Then $A\not\subset L$ or $A\not\subset L'$.
Suppose that $A\not\subset L$.
Consider the point $b\in M\cap L\setminus\{a\}$
closest to $a$.
The existence of a third point in $M$ on any line $\overline{b\alpha}$
with $\alpha\in A\setminus\{a\}$
yields $M\setminus \bigcup\mathcal{F}$ infinite, absurd.
Hence, there exists another line-segment $A'\in \mathcal{F}$
with $a$ as an endpoint.

If no other line-segment has $a$ as an endpoint, then the
existence of a third point of $M$ on each line $\overline{\alpha\alpha'}$
with $\alpha\in A\setminus\{a\}, \alpha'\in A'\setminus\{a\}$,
both close to $a$,
yields $M\setminus\bigcup\mathcal{F}$ infinite, absurd again.

Hence, there are 3 line-segments with $a$ as an endpoint,
and the Claim is proven.

Let $\mathcal{P}$ be the set of sides of $P$.
We have $\mathcal{P}\subset \mathcal{F}$.
Indeed, assume $ab$ is a side of the polygon $P$ not in $\mathcal{F}$.
Then line-segments from $\mathcal{F}$
with $a,b$ as endpoints, respectively, are distinct.
Following the Claim, there are 3 line-segments in $\mathcal{F}$
with an endpoint in $a$, and another $3$ line-segments with
an endpoint in $b$. Thus, $\mathrm{card}\mathcal{F}\geq 6$,
a contradiction.

Therefore, $\mathrm{card}\mathcal{P}\leq 5$.
Let $A, ab, B$ be consecutive sides of $P$.
Another line-segment of $\mathcal{F}$ will have
$a$ as an endpoint, the argument being as above.

Hence, for every vertex $v$ of $P$, there is a line-segment
in $\mathcal{F}\setminus \mathcal{P}$, with an endpoint
at $v$. This yields $\mathrm{card}\mathcal{F}\geq 6$,
if $P$ is a triangle or a quadrilateral,
and $\mathrm{card}\mathcal{F}\geq 8$ if $P$ is a pentagon.
A contradiction is obtained.
\end{proof}

\begin{figure}[htbp]
\centering
\includegraphics[width=9cm]{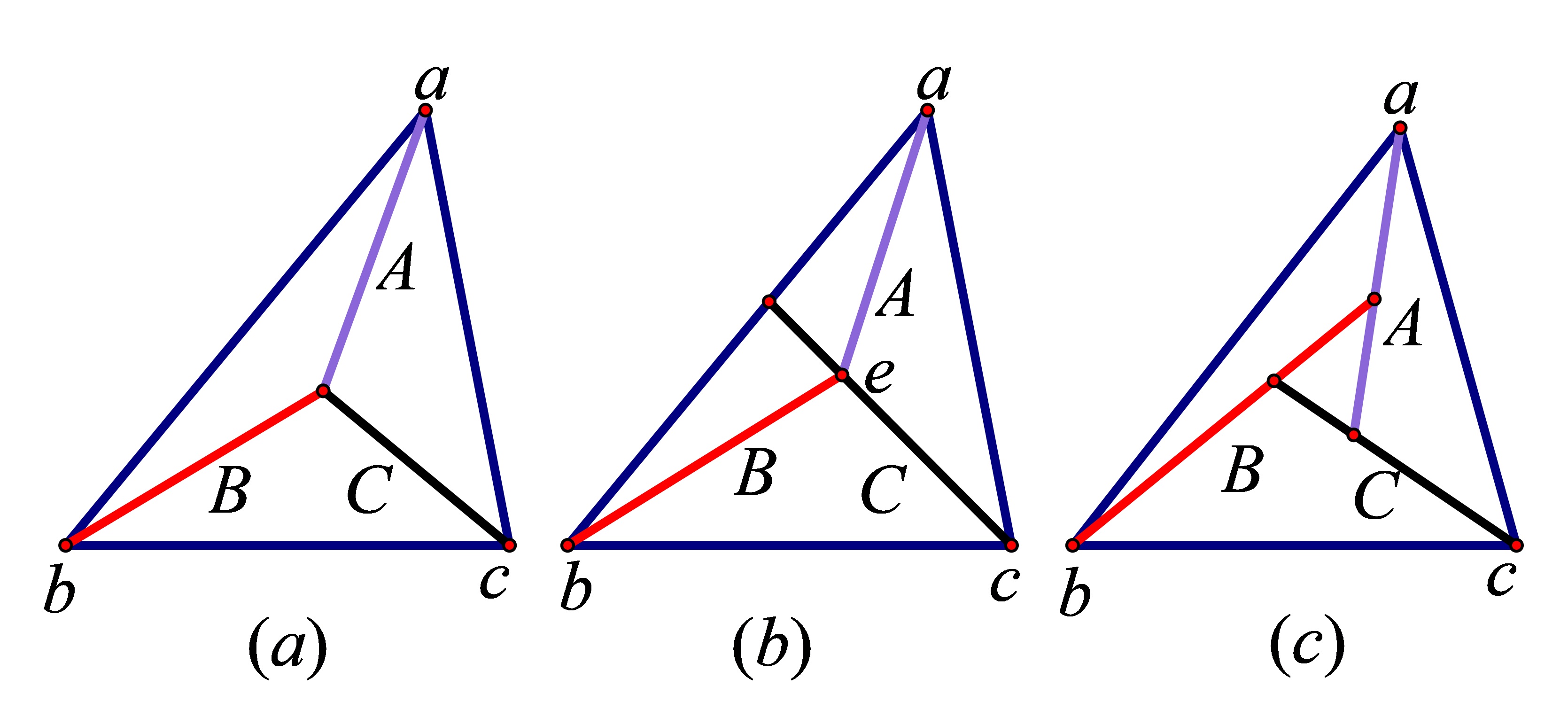}\\
\includegraphics[width=12cm]{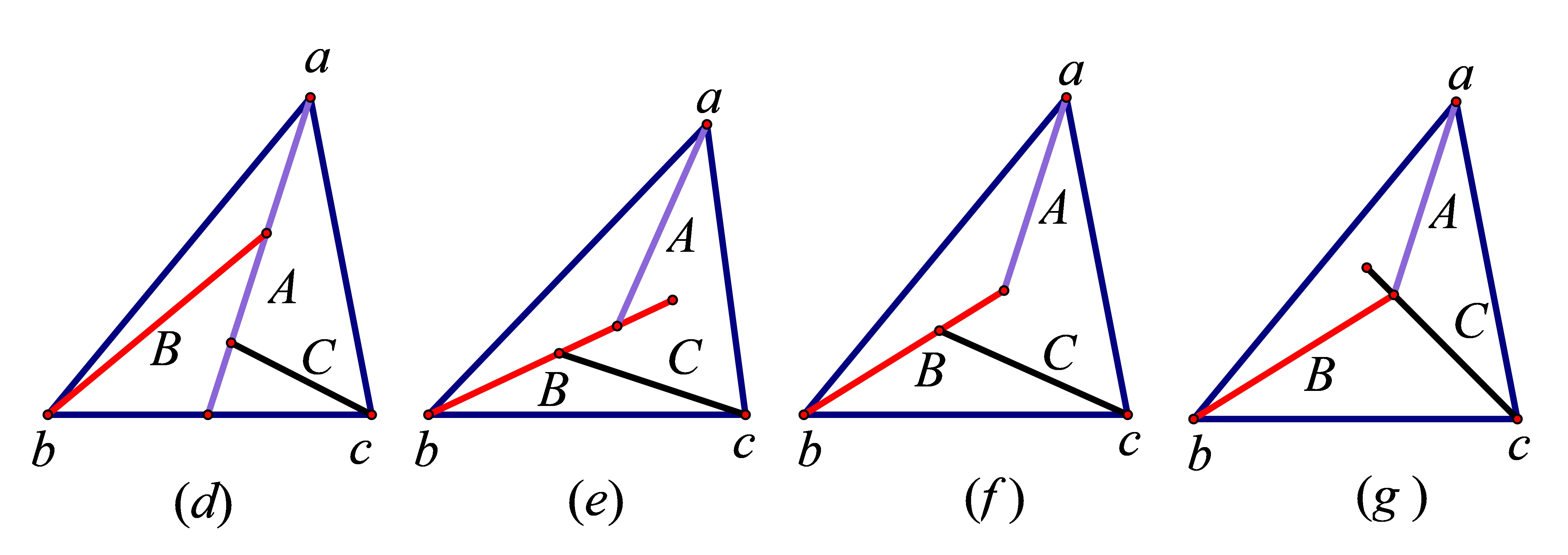}
\caption{Counterexamples with $\mathrm{card}\mathcal{F}=6$.}
\label{coun-6}
\end{figure}

Notice that $5$ is  best possible in Theorem \ref{line-Sym}.
For $\mathrm{card}\mathcal{F}=6$, there are counterexamples.
See Figure \ref{coun-6}.
$\bigcup \mathcal{F}$ can be regarded as a planar geometric graph.
As graphs, only 4 of them are 3-connected (see Figure \ref{coun-6} $(a)-(d)$), which will be shown later.
In general, they can have connectivity 1, 2, or 3.


The Claim from the preceding proof also holds for any finite family of line-segments.

\begin{lemma}\label{3-segments}
Let $\mathcal{F}$ be a finite family of line-segments
with pairwise disjoint relative interiors in $\mathbb{R}^2$.
If~~$\bigcup\mathcal{F}$ is not collinear
and meets no line in exactly two points,
then, for each vertex of $\mathrm{conv} \bigcup\mathcal{F}$,
there exist three line-segments from $\mathcal{F}$ having
an endpoint there.
\end{lemma}

Next, we utilize Lemma \ref{3-segments} to prove the following.

\begin{theorem}\label{coun-6-thm}
Let $\mathcal{F}$ be a family of $6$ line-segments with pairwise
disjoint relative interiors in $\mathbb{R}^2$.
If no line meets $\bigcup\mathcal{F}$
in exactly two points and
$\bigcup\mathcal{F}$ is $3$-connected as a graph,
then
it is of one of the four combinatorial types shown in
Figure \ref{coun-6} $(a)-(d)$.
\end{theorem}

\begin{proof}
Since $\bigcup\mathcal{F}$ is 3-connected,
it is not collinear.
Then $P=\mathrm{conv}\bigcup\mathcal{F}$ is a polygon (which is not a line-segment).
Let $\mathcal{P}$ be the set of sides of $P$.
By Lemma \ref{3-segments},
there are 3 line-segments from $\mathcal{F}$
with an endpoint at each vertex of $P$.
If $\mathrm{card}\mathcal{P}\geq 4$, then
$\mathrm{card}\mathcal{F}$ must be at least 7,
which contradicts the given condition
$\mathrm{card}\mathcal{F}=6$.
So $P$ is a triangle $abc$.
Since there are 3 line-segments from $\mathcal{F}$
with an endpoint at each vertex of $P$ and $\mathrm{card}\mathcal{F}=6$,
all sides of $P$ must be included in $\mathcal{F}$.
Denote the remaining line-segments in $\mathcal{F}$
by $A, B, C$, with endpoints at $a, b, c$, respectively.
Notice that, due to the fact that $\bigcup\mathcal{F}$ is 3-connected,
the degree of each endpoint of a line-segment from $\mathcal{F}$
is at least $3$.

If $A,B,C$ have a common endpoint, then $\bigcup\mathcal{F}$
is like in Figure \ref{coun-6} $(a)$.

If two of $A,B,$ and $C$, say $A$ and $B$, have a common endpoint $e$,
then $C$ must pass through $e$ to reach $ab$,
because the degree of each vertex of $\bigcup\mathcal{F}$
is at least $3$.
That is the case shown in Figure \ref{coun-6} $(b)$.


If no two of $A,B$ and $C$ share an endpoint,
without loss of the generality,
assume that $B\cap \mathrm{int}A\neq \emptyset$.
Then there are two cases to be considered.
If $C\cap \mathrm{int}B\neq\emptyset$,
then we have the case depicted in Figure \ref{coun-6} $(c)$.
If $C\cap \mathrm{int}B=\emptyset$,
then $A\cap \mathrm{int}bc\neq\emptyset$ and
$C\cap \mathrm{int}A\neq \emptyset$.
Then $\bigcup \mathcal{F}$ is the graph
shown in Figure \ref{coun-6} $(d)$.
\end{proof}

Notice that the graphs in
Figure \ref{coun-6} $(c)$ and Figure \ref{coun-6} $(d)$
are isomorphic (with the 1-skeleton of a triangular prism).

\begin{figure}[htbp]
\centering
\includegraphics[width=8cm]{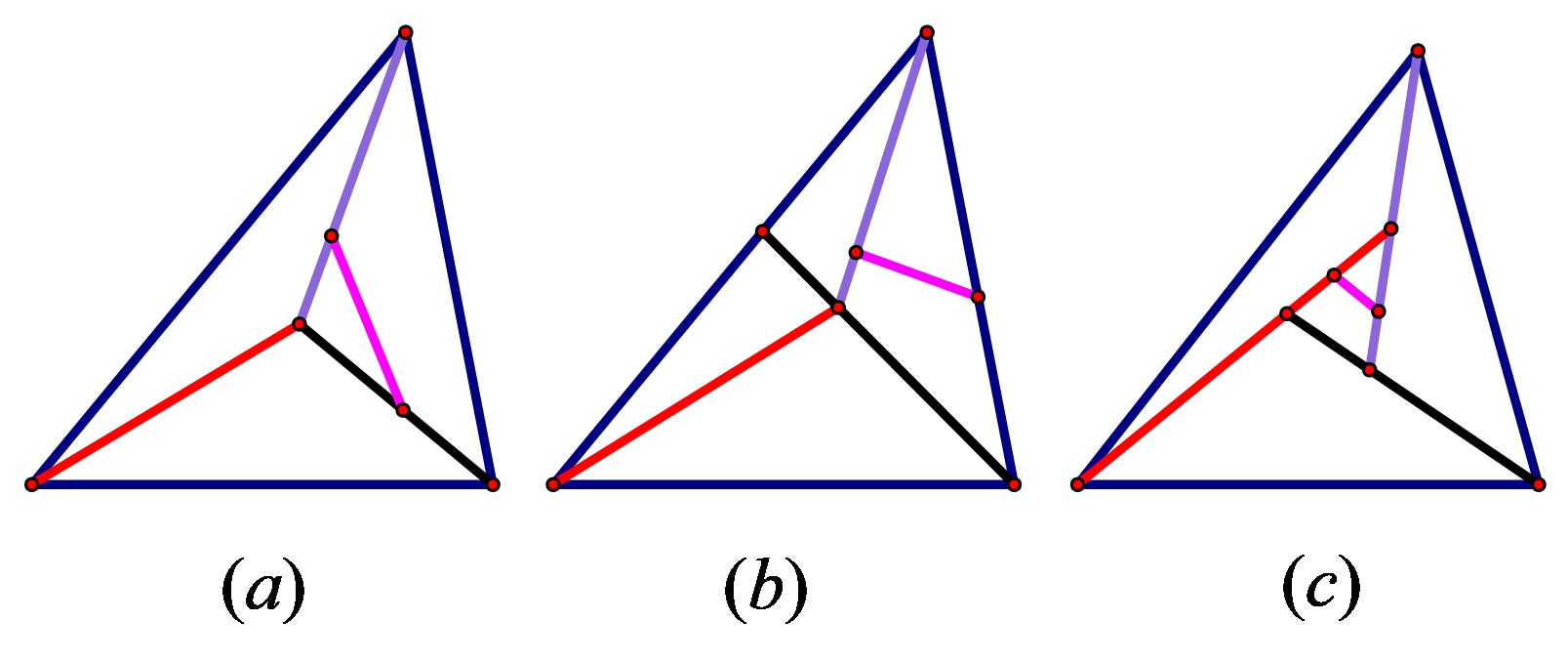}\\
\includegraphics[width=11cm]{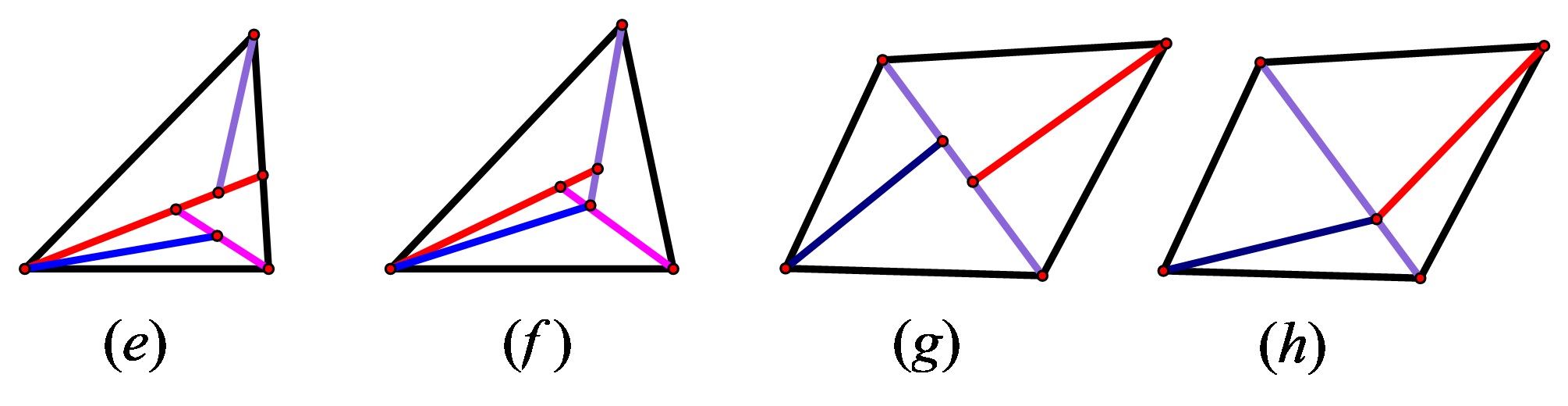}
\caption{$\mathrm{card}\mathcal{F}=7$ and $\bigcup\mathcal{F}$ is $3$-connected.}
\label{coun-7-1}
\end{figure}

For $\mathrm{card}\mathcal{F}=7$,
there are many combinatorial types of 3-connected graphs $\bigcup\mathcal{F}$.
For example, the graphs shown in Figure \ref{coun-7-1} $(a)-(c)$
can be obtained by adding
a line-segment in Figure \ref{coun-6}.
Other graphs are shown in Figure \ref{coun-7-1} $(e)-(h)$.

\section*{Acknowledgements}

This work is supported by NSF of China (12271139, 12201177); the Hebei Natural Science Foundation (A2024205012, A2023205045); the Special Project on Science and Technology Research and Development Platforms, Hebei Province (22567610H);
the Science and Technology Project of Hebei Education Department (BJK2023092); the Foreign Experts Program of People's Republic of China; the Program for Foreign Experts of Hebei Province. The first author was partially supported by NKFIH grant No. 133819 and also by the HUN-REN Research Network.

\end{document}